\documentclass[12pt,reqno]{amsart}
\usepackage{amsfonts,amscd,latexsym,amsmath,amssymb,cancel,enumerate,mathrsfs,hyperref}
\theoremstyle{plain}
\newtheorem{master}{Master}[section]
\newtheorem{prop}[master]{Proposition}
\newtheorem{thm}[master]{Theorem}
\newtheorem{fact}[master]{Fact}
\newtheorem{lem}[master]{Lemma}

\newtheorem{question}[master]{Question}

\newtheorem{claim}[master]{Claim}
\theoremstyle{definition}
\newtheorem{defin}[master]{Definition}

\theoremstyle{remark}

\numberwithin{equation}{section}

\newcommand{\Nat}{\mathbb{N}}
\begin{document}
\title{Scott rank of Polish metric spaces}
\author{Michal Doucha}
\address{Institute of Mathematics, Polish Academy of Sciences, Warsaw, Poland}
\email{m.doucha@post.cz}

\thanks{The author was supported by funds allocated to the implementation of the international co-funded project in the years 2014-2018, 3038/7.PR/2014/2, and by the EU grant PCOFUND-GA-2012-600415.}
\begin{abstract}
Following the work of Friedman, Koerwien, Nies and Schlicht we positively answer their question whether the Scott rank of Polish metric spaces is countable.
\end{abstract}
\maketitle
\section*{Introduction}
The origins of infinitary logic go back to the 1960s when it emerged through the work of Carp, Scott, Morley, Lopez-Escobar among others. Let us highlight here mainly the work of Scott on countable structures (\cite{Sc}): \emph{Let $L$ be a countable language and $M$ a countable $L$-structure. Then there is a sentence $\varphi$ of the $L_{\omega_1,\omega}$-logic such that $M\models \varphi$, and if $N$ is another countable $L$-structure such that $N\models \varphi$, then $M\cong N$.}

Tightly connected is the notion of `Scott rank' that will be defined below.\\

Recently, by Friedman, Koerwien, Nies, Schlicht (see \cite{Ni} and \cite{FFKN}), the Scott analysis has been applied to the case of Polish metric spaces to better understand the equivalence relation of isometry between Polish metric spaces. We recall the result of Clemens, Gao and Kechris from \cite{CGK} where they prove that the isometry equivalence relation between Polish metric spaces is Borel bireducible with the universal orbit equivalence relation induced by a Polish action of a Polish group. In particular, it is an analytic non-Borel equivalence relation where every equivalence class is Borel. Since it is possible to consider metric spaces as relational structures in a first order countable language one can apply the Scott analysis there. Recall that metric space $(X,d)$ can be viewed as a structure in a countable relational language if we define binary relations $d_q$ and $d^q$ for every $q\in \mathbb{Q}^+$ and interpret $d_q(x,y)$ as $d(x,y)<q$ and $d^q(x,y)$ as $d(x,y)>q$.

Although Polish metric spaces are uncountable in general they share some properties with countable structures. For instance, if two Polish metric spaces are elementary equivalent in $L_{\infty \omega}$ logic (we refer to \cite{Ke} for any unexplained notion from infinitary logic) then they are isometric. However, there are some properties that distinguish countable structures from Polish metric spaces. For example, as was pointed out by Kechris, there is no Borel assignment giving to every Polish metric space $X$ an $L_{\omega_1 \omega}$ sentence $\varphi_X$ such that for any other Polish metric space $Y$ we have $Y\models \varphi_X$ iff $X\cong_{iso} Y$. Otherwise, one would get that the isometry equivalence relation between Polish metric spaces is classifiable by countable structures, which contradicts the result of Clemens, Gao and Kechris (\cite{CGK}).

Scott rank is the common measurement of model-theoretic complexity. All countable structures have countable Scott rank, whereas from the general theory one can deduce that every Polish metric space has a Scott rank of cardinality at most continuum. Friedman, Koerwien, Nies and Schlicht proved that there exist Polish metric spaces of arbitrarily high countable Scott rank. The following question of Nies was left open.
\bigskip

\begin{question}
Is the Scott rank of any Polish metric countable?
\end{question}
\bigskip

Here we answer the question in affirmative. In the following section we review all necessary notions from infinitary logic used in the paper. We refer to \cite{Ke} for a general reference on this subject. We also refer to \cite{Ho} for the connections of the Scott rank and the Ehrenfeucht-Fra\" iss\' e games and to \cite{Gao} for another explanation of Scott analysis connected with descriptive set theory and Polish metric spaces.
\section{Preliminaries}
\begin{defin}
Let $L$ be a countable relational language. Let $M$ be a structure of the language $L$. Let $(\vec{a},\vec{b})$ be a pair of (ordered) tuples of the same length from $M$. We write
\begin{enumerate}
\item $\vec{a}\equiv_0 \vec{b}$ if there is a partial isomorphism which maps $\vec{a}$ onto $\vec{b}$
\item If $\alpha$ is a limit ordinal then we write  $\vec{a}\equiv_\alpha \vec{b}$ if $\vec{a}\equiv_\beta \vec{b}$ for all $\beta<\alpha$.
\item Finally, if $\alpha=\beta+1$, for some ordinal $\beta$, then we write $\vec{a}\equiv_\alpha \vec{b}$ if for every $x_a,x_b\in M$ there are elements $y_b,y_a\in M$ such that $\vec{a}x_a\equiv_\beta \vec{b}y_b$ and $\vec{a}y_a\equiv_\beta \vec{b}x_b$.

\end{enumerate}
Now, for every such pair $(\vec{a},\vec{b})$ of tuples from $M$, define the Scott rank of $(\vec{a},\vec{b})$, $\mathrm{sr}((\vec{a},\vec{b}))$, as $-1$ if $\vec{a}\equiv_\alpha \vec{b}$ for every ordinal $\alpha$; otherwise, we define it as $\inf\{\alpha: \vec{a}\cancel{\equiv}_\alpha \vec{b}\}$.

Finally, we define the Scott rank of $M$, $\mathrm{sr}(M)$, as $\sup\{\mathrm{sr}(\vec{a},\vec{b})+1:(\vec{a},\vec{b})\text{ are tuples of the same arbitrarily large length from }M\}$.
\end{defin}
The following fact is folklore and easy to check.
\begin{fact}\label{basic}
Let $M$ be a structure in some countable language $L$. Then $\mathrm{sr}(M)<|M|^+$.
\end{fact}
If $(X,d)$ is a Polish metric space then it follows that we have that $\mathrm{sr}(X)<\mathfrak{c}^+$.

Let us present another characterization of the $\equiv_\alpha$ relation that will be used in the proof of the main theorem. We start with recalling the definition of Ehrenfeucht-Fra\" iss\' e game (of length $\alpha$).\\

\noindent{\bf Ehrenfeucht-Fra\" iss\' e game}\\
Let again $M$ be some structure of a countable language $L$, $\alpha$ some ordinal and let $(\vec{a},\vec{b})$ be a pair of tuples of the same length from $M$ for which there exists a partial isomorphism of $M$ mapping one to the other. At the first step Player I chooses an ordinal $\alpha_1<\alpha$, the left or right side and an element from $M$ denoted either by $x_1^L$ (if Player I chose the left side) or $x_1^R$ (if he chose the right side). Player II responds by playing an element of $M$ denoted by $x_1^R$ (if Player I played $x_1^L$) or $x_1^L$ (if Player I played $x_1^R$). In the next round, Player I chooses an ordinal $\alpha_2<\alpha_1$, again the left or right side and an element from $M$ denoted either by $x_2^L$ (if Player I chose the left side) or $x_2^R$ (if he chose the right side). Player II responds by playing an element of $M$ denoted by $x_2^R$ (if Player I played $x_2^L$) or $x_2^L$ (if Player I played $x_2^R$). They continue similarly until Player I plays $0$ as an ordinal. Thus the game ends after finitely many (let us say $n$) rounds. Player II wins if there is a partial isomorphism of $M$ mapping $\vec{a}x_1^L\ldots x_n^L$ onto $\vec{b}x_1^R\ldots x_n^R$. Otherwise, Player I wins.

Let us denote such a game by $\mathrm{EF}(\vec{a},\vec{b},\alpha)$.\\

By $\mathrm{EF}(\vec{a},\vec{b},\infty)$ we denote the game which is played analogously with the exception that Player I does not choose an ordinal but only plays a side and an element. The game ends after countably many steps in which Player I and II produce elements $(x_n^L)_{n\in \Nat}$ and $(x_n^R)_{n\in \Nat}$, and Player II wins if there is a partial isomorphism of $M$ mapping $\vec{a}$ onto $\vec{b}$, and for every $n\in \Nat$ mapping $x_n^L$ to $x_n^R$. Otherwise, Player I wins.\\

The following fact will be used in the proof of the main theorem.
\begin{fact}\label{EFcharac}
Let $M$ be as above and let $(\vec{a},\vec{b})$ be a pair of tuples of the same length from $M$. Let $\alpha$ be some ordinal. Then $\vec{a}\equiv_\alpha \vec{b}$ iff Player II has a winning strategy in $\mathrm{EF}(\vec{a},\vec{b},\alpha)$.

Moreover, if $M$ is either a countable structure or a Polish metric space, then there is an automorphism of $M$ mapping $\vec{a}$ onto $\vec{b}$ iff Player II has a winning strategy in $\mathrm{EF}(\vec{a},\vec{b},\infty)$.
\end{fact}
\begin{proof}
We proceed by induction on $\alpha$. The case $\alpha=0$ is clear.

Assume that $\alpha$ is limit and we have proved the statement for all ordinals less than $\alpha$. Suppose that $\vec{a}\equiv_\alpha \vec{b}$. We must produce a winning strategy $\phi_\alpha$ for Player II in $\mathrm{EF}(\vec{a},\vec{b},\alpha)$. By inductive assumption, there is a winning strategy $\phi_\beta$, for every $\beta<\alpha$, for Player II in $\mathrm{EF}(\vec{a},\vec{b},\beta)$. Player I chooses an ordinal $\alpha_1<\alpha$ in his first move in $\mathrm{EF}(\vec{a},\vec{b},\alpha)$. We can let Player II play following the strategy $\phi_{\alpha_1+1}$. This describes the strategy $\phi_\alpha$.

To prove the other direction, suppose that Player II has a winning strategy $\phi_\alpha$ in $\mathrm{EF}(\vec{a},\vec{b},\alpha)$. We want to prove that $\vec{a}\equiv_\alpha \vec{b}$. Clearly, $\phi_\alpha$ is a winning strategy also in $\mathrm{EF}(\vec{a},\vec{b},\beta)$, for every $\beta<\alpha$. By inductive assumption, we have that $\vec{a}\equiv_\beta \vec{b}$ for every $\beta<\alpha$. Thus $\vec{a}\equiv_\alpha \vec{b}$.

Assume now that $\alpha=\beta+1$ and we have proved the statement for $\beta$. Suppose that $\vec{a}\equiv_\alpha \vec{b}$. We must again produce a winning strategy $\phi_\alpha$ for Player II in $\mathrm{EF}(\vec{a},\vec{b},\alpha)$. By inductive assumption, we already have a winning strategy $\phi_\beta$ for Player II in $\mathrm{EF}(\vec{a},\vec{b},\beta)$. Player I plays an ordinal $\alpha_1<\alpha$ in his first move in $\mathrm{EF}(\vec{a},\vec{b},\alpha)$. Suppose that $\alpha_1<\beta$. Then we may let Player II continue playing by following the strategy $\phi_\beta$. So suppose that $\alpha_1=\beta$ and that he plays $x_1^L\in M$ (i.e. chooses the left side). By definition, there exists $x_1^R$ such that $\vec{a}x_1^L\equiv_\beta \vec{b}x_1^R$. We let Player II respond by playing this $x_1^R$ in his first move and then following the strategy $\phi'_\beta$, where $\phi'_\beta$ is a winning strategy for Player II in $\mathrm{EF}(\vec{a}x_1^L,\vec{b}x_1^R,\beta)$, which exists by inductive assumption. The case when Player I chooses the right side in his first move is analogous.

We now prove the other direction. Suppose that Player II has a winning strategy $\phi_\alpha$ in $\mathrm{EF}(\vec{a},\vec{b},\alpha)$. We want to prove that $\vec{a}\equiv_\alpha \vec{b}$. Let $x_a,x_b\in M$ be arbitrary. We have to find $y_a,y_b\in M$ such that $\vec{a}x_a\equiv_\beta \vec{b}y_b$ and $\vec{a}y_a\equiv_\beta \vec{b}x_b$. We will show how to find $y_b$. We let Player I play $\beta$ and $x_1^L=x_a$ (i.e. the left side) in his first move. Let $y_b$ be the response of the strategy $\phi_\alpha$. The strategy $\phi_\alpha$, in the next rounds, behaves like a winning strategy for Player II in $\mathrm{EF}(\vec{a}x_a,\vec{b}y_b,\beta)$. Thus by inductive assumption, we have $\vec{a}x_a\equiv_\beta \vec{b}y_b$ and we are done.\\

The latter assertion from the statement of the fact is folklore for countable structures and the proof may be found in \cite{Ho} (Theorem 3.2.3). For Polish metric spaces, one can use the same argument to produce an isometry between two countable dense subsets and then extend it to an autoisometry of the whole space.
\end{proof}

Let us conclude this section by some notions related to stationary subsets of ordinals. We refer the reader to \cite{Je} (Chapter 8) for basic information about stationary and non-stationary sets. Recall that for any non-stationary subset of $\omega_1$ there exists a regressive non-decreasing function such that the preimages of singletons are bounded (or equivalently, the set of values of this function is uncountable). We will call such functions \emph{non-stationary} as they can be defined only on non-stationary subsets (because of Fodor's lemma).
\section{Main theorem}
\begin{thm}\label{maintheorem}
Let $(X,d)$ be a Polish metric space. Then the Scott rank of $X$ is countable.
\end{thm}
The proof of Theorem \ref{maintheorem} is divided into two steps: Proposition \ref{firststep} and Proposition \ref{secondstep}. For the rest of this section, let $(t_n)_{n\in \Nat}\subseteq X$ be a sequence of points of $X$ with the following property: $$\forall x\in X\forall \varepsilon>0\forall n_0\in \Nat \exists n\geq n_0 (d(x,t_n)<\varepsilon)$$
For instance, the enumeration of some countable dense subset of $X$ where each isolated element of this dense set appears infinitely often is an example.

\begin{prop}\label{firststep}
Let $(\vec{a},\vec{b})$ be a pair of finite tuples of the same length from $X$. Then the Scott rank of $(\vec{a},\vec{b})$ is countable.
\end{prop}
It follows from Proposition \ref{firststep} that $\mathrm{sr}(X)\leq \omega_1$. If it were greater than $\omega_1$ there would have to be a pair of tuples $(\vec{a},\vec{b})$ of the same length from $X$ such that $\mathrm{sr}(\vec{a},\vec{b})\geq \omega_1$ which would contradict the proposition.\\
\begin{proof}
To simplify the notation we shall assume that the length of the tuples is $1$, i.e. we have a pair of points $(a,b)$. Suppose that for every $\alpha<\omega_1$ we have $a\equiv_\alpha b$. We shall prove that then there is an isometry of $X$ which maps $a$ to $b$. This will imply that $a\equiv_\alpha b$ for every $\alpha$.

Using Fact \ref{EFcharac}, for every $\alpha<\omega_1$ we have some  winning strategy $\phi_\alpha$ for Player II in the game $\mathrm{EF}(a,b,\alpha)$. In what follows, we will play the Ehrenfeucht-Fra\" iss\' e games $\mathrm{EF}(a,b,\alpha)$, for each $\alpha<\omega_1$, simultaneously. We shall inductively construct functions $\psi,\varphi:\Nat\rightarrow \Nat$, uncountable subsets $N_1\supseteq M_1\supseteq N_2\supseteq M_2\supseteq\ldots$ of $\omega_1$ and non-stationary functions $F_1:N_1\rightarrow \omega_1$, $G_1:F_1[M_1]\rightarrow \omega_1$, $F_2:G_1\circ F_1[N_2]\rightarrow \omega_1$, etc. The functions $\psi,\varphi$ will determine functions $t_n\to t_{\psi(n)}$ and $t_n\to t_{\varphi(n)}$ that will be `almost-isometries' and will help us define an autoisometry of $X$ mapping $a$ to $b$. The role of the non-stationary functions will be to prescribe which ordinals Player I should play in his rounds: Each game $\mathrm{EF}(a,b,\alpha)$ ends after finitely many rounds. Playing carefully using the non-stationary functions we can however guarantee that after finitely many rounds we still have uncountably many ordinals $\alpha$ such that we can continue playing in the game $\mathrm{EF}(a,b,\alpha)$.

Since the first and the general step of the induction is basically the same, we will describe it at once.\\

\noindent {\bf The $n$-th step of the induction.} Suppose that we have already found $y^1,x^1,\ldots,y^{n-1},x^{n-1}$, uncountable sets $N_1\supseteq M_1\supseteq \ldots M_{n-1}\supseteq N_n$ and non-stationary functions $F_1,G_1,\ldots, G_{n-1},F_n$ such that for every $i\leq n-1$ both $F_i\circ G_{i-1}\circ \ldots\circ F_1[M_i]$ and $G_i\circ F_i\circ\ldots\circ F_1[N_{i+1}]$ are uncountable non-stationary subsets. In case we are at the first step of the induction, we just consider some arbitrary uncountable non-stationary subset $N_1\subseteq \omega_1$ and some non-stationary function $F_1:N_1\rightarrow \omega_1$.

For each $\alpha\in N_n$ consider the response of the strategy $\phi_\alpha$ when Player I plays $(F_1(\alpha),L,t_1)$ in the first round, $(G_1\circ F_1(\alpha),R,t_1)$ in the second round, so on, and finally plays $(F_n\circ G_{n-1}\circ\ldots\circ F_1(\alpha),L,t_n)$. For the case $n=1$, it means we just consider the response of the strategy $\phi_\alpha$ when Player I plays $(F_1(\alpha),L,t_1)$.

Denote such a response by $y_\alpha^n$. Observe that since $\phi_\alpha$ is a winning strategy for Player II we have
\begin{equation}\label{addedeq1}
d(a,t_n)=d(b,y_\alpha^n)
\end{equation}
and, if $n>1$, for any $m\leq n-1$
\begin{equation}\label{addedeq2}
d(t_m,t_n)=d(y_\alpha^m,y_\alpha^n)
\end{equation}
Since $F_n\circ G_{n-1}\circ\ldots\circ F_1[N_n]$ is uncountable and $X$ is separable there exists an uncountable subset $M'_n\subseteq N_n$ such that
\begin{equation}\label{firsteq} 
\forall \alpha,\beta\in M'_n (d(y_\alpha^n,y_\beta^n)<1/2^n)
\end{equation}
Let $\psi(n)\in \Nat$ be an arbitrary natural number such that for some $\alpha\in M'_n$ we have $d(t_{\psi(n)},y_\alpha^n)<1/2^n$, e.g. $\min\{m\in \Nat: \exists \alpha\in M'_n (d(t_m,y_\alpha^n)<1/2^n\}$. Shrink $M'_n$, if necessary, to an uncountable subset $M_n\subseteq M'_n$ so that $F_n\circ\ldots\circ F_1[M_n]$ is non-stationary. Let $G_n:F_n\circ\ldots\circ F_1[M_n]\rightarrow \omega_1$ be some non-stationary function.

Again, for each $\alpha\in M_n$ consider the response of the strategy $\phi_\alpha$ when Player I plays $(F_1(\alpha),L,t_1)$ in the first round, $(G_1\circ F_1(\alpha),R,t_1)$ in the second round, so on, and finally plays $(F_n\circ G_{n-1}\circ\ldots\circ F_1(\alpha),L,t_n)$ in the $2n-1$-st round and then $(G_n\circ F_n\circ\ldots\circ F_1(\alpha),R,t_n)$ in the $2n$-th round. Denote such a response by $x_\alpha^n$. Again observe that since $\phi_\alpha$ is a winning strategy for Player II we have
\begin{equation}\label{addedeq3}
d(b,t_n)=d(b,x_\alpha^n)
\end{equation}
and, if $n>1$, for any $m\leq n-1$
\begin{equation}\label{addedeq4}
d(t_m,t_n)=d(x_\alpha^m,x_\alpha^n)
\end{equation}
Since $G_n\circ\ldots\circ F_1[M_n]$ is uncountable and $X$ is separable there exists an uncountable subset $N'_{n+1}\subseteq M_n$ such that
\begin{equation}\label{secondeq}
\forall \alpha,\beta\in N'_{n+1} (d(x_\alpha^n,x_\beta^n)<1/2^n)
\end{equation}
Let $\varphi(n)\in \Nat$ be an arbitrary, e.g. minimal, element of the set $\{m\in \Nat:\exists \alpha\in N'_{n+1} (d(t_m,x_\alpha^n)<1/2^n)\}$. Shrink $N'_{n+1}$, if necessary, to an uncountable subset $N_{n+1}\subseteq N'_{n+1}$ so that $G_n\circ\ldots\circ F_1[N_{n+1}]$ is non-stationary. Finally, let $F_{n+1}:G_n\circ\ldots\circ F_1[N_{n+1}]\rightarrow \omega_1$ be some non-stationary function.

This finishes the inductive construction.\\

After the construction is done, we have the functions $\psi,\varphi:\Nat\rightarrow \Nat$.
\begin{claim}\label{goodpsi}
Let $i,j\in \Nat$. Let $\rho(i)$ denote either $\psi(i)$ or any element from $\varphi^{-1}(i)$ (provided $\varphi^{-1}(i)$ is non-empty), $\rho(j)$ is defined analogously.   Then we have $$|d(t_i,t_j)-d(t_{\rho(i)},t_{\rho(j)})|<1/2^{i-1}+1/2^{j-1}$$ and $$|d(a,t_j)-d(b,t_{\rho(j)})|<1/2^{j-1}$$
\end{claim}
\noindent\emph{Proof of the claim.} Fix some $i,j$ and suppose that $i<j$ and also that $\rho(i)=\psi(i)$ and $\rho(j)=\psi(j)$; the other cases are analogous and omitted. Let $\alpha\in M'_j$, by (\ref{firsteq}) (and definition of $\psi(j)$ below (\ref{firsteq})) we have $$d(t_{\rho(j)},y_\alpha^j)<1/2^{j-1}$$ and $$d(t_{\rho(i)},y_\alpha^i)<1/2^{i-1}$$ Since $\phi_\alpha$ is a winning strategy, by (\ref{addedeq2}) we must have $$d(t_i,t_j)=d(y_\alpha^i,y_\alpha^j)$$ Thus putting the (in)equalities above together we get $$|d(t_i,t_j)-d(t_{\rho(i)},t_{\rho(j)})|<1/2^{i-1}+1/2^{j-1}$$

Similarly, since $\phi_\alpha$ is a winning strategy, by (\ref{addedeq1}) we must have $$d(a,t_j)=d(b,y_\alpha^j)$$ and the inequality $$|d(a,t_j)-d(b,t_{\rho(j)})|<1/2^{j-1}$$ again follows.

\hfill $\qed$ (of the claim)\\

We now define the autoisometry $\chi:X\rightarrow X$ taking $a$ to $b$. For any $x\in X$ choose arbitrarily some strictly increasing $\iota:\Nat\rightarrow \Nat$ so that the sequence $(t_{\iota(n)})_n$ is a Cauchy sequence converging to $x$. It follows from Claim \ref{goodpsi} that $(t_{\psi(\iota(n))})_n$ is Cauchy as well and we set $\chi(x)$ as the limit of this sequence. It also follows from Claim \ref{goodpsi} that $\chi$ is correctly defined, i.e. it does not matter which strictly increasing $x:\Nat\rightarrow \Nat$ with the desired properties we choose, that $\chi(a)=b$, and finally that $\chi$ is an autoisometry of $X$. This finishes the proof.
\end{proof}

Thus we must rule out the possibility that $\mathrm{sr}(X)=\omega_1$. Assume that it is the case. Then there exists a cofinal subset $A\subseteq \omega_1$ of countable ordinals such that for every $\alpha\in A$ there is a pair of tuples of the same length $(\vec{a}_\alpha,\vec{b}_\alpha)$ from $M$ such that $\vec{a}_\alpha\equiv_\alpha \vec{b}_\alpha$, however $\vec{a}_\alpha \cancel{\equiv}_{\alpha+1} \vec{b}_\alpha$. Without loss of generality, we may suppose that the length of $\vec{a}_\alpha$ and $\vec{b}_\alpha$ is the same, say $n$, for all $\alpha\in A$.

For a set of (countable) ordinals $A$ and some $n\in \Nat$, let us call the $A$-indexed set of pairs of tuples $((\vec{a}_\alpha,\vec{b}_\alpha))_{\alpha\in A}$ an $(A,n)$-family if for every $\alpha\in A$ we have $\vec{a}_\alpha\equiv_\alpha \vec{b}_\alpha$ and the length of both $\vec{a}_\alpha$ and $\vec{b}_\alpha$ is $n$.

We reach the contradiction by applying the following proposition.
\begin{prop}\label{secondstep}
Let $A\subseteq \omega_1$ be a cofinal subset of countable ordinals and $n\in\Nat$. Let $((\vec{a}_\alpha,\vec{b}_\alpha))_{\alpha\in A}$ be an $(A,n)$-family. Then there exists an uncountable subset $B\subseteq A$ such that for every pair $(\vec{a},\vec{b})$ lying in the closure of $\{(\vec{a}_\alpha,\vec{b}_\alpha):\alpha\in B\}\subseteq X^{2n}$ there exists an autoisometry of $X$ mapping $\vec{a}$ onto $\vec{b}$.
\end{prop}
Observe that once the proposition is proved, we are done. Indeed, apply Proposition \ref{secondstep} to the set $A$ above to get the set $B$. Then for every $\alpha\in B$, the pair $(\vec{a}_\alpha,\vec{b}_\alpha)$ obviously lies in the closure of $\{(\vec{a}_\alpha,\vec{b}_\alpha):\alpha\in B\}\subseteq X^{2n}$, thus, according to the propositon, there exists an autoisometry of $X$ mapping $\vec{a}_\alpha$ onto $\vec{b}_\alpha$. That is, however, in contradiction with the assumption that $\vec{a}_\alpha \cancel{\equiv}_{\alpha+1}\vec{b}_\alpha$. Eventhough, it would be sufficient to reach the contradiction with a single such pair, it is not clear how to find it without actually showing that there are uncountably many such pairs.\\

Thus it remains to prove Proposition \ref{secondstep}. In what follows, when $\kappa$ is an infinite cardinal and $A\subseteq B$ are two subsets of some topological space, then we shall say that $A$ is {\it $\kappa$-dense} in $B$ if for every relatively open subset $O\subseteq B$ we have $|A\cap O|\geq \kappa$.
\begin{proof}[Proof of Proposition \ref{secondstep}]
For any subset $C\subseteq A$, by $X_C$ we shall denote the set $\{(\vec{a}_\alpha,\vec{b}_\alpha):\alpha\in C\}$. Let $B\subseteq A$ be an uncountable subset such that for every basic open set $O\subseteq X^{2n}$ the intersection $X_B\cap O$ is either uncountable or empty. Such $B$ exists since $X$ is second countable, i.e. there are only countably many basic open sets.

Let $F$ be the closed set that is a complement of the following open set: $\bigcup \{O:O\text{ is basic open and }O\cap X_B=\emptyset\}$. By assumption, $X_B$ is $\omega_1$-dense in $F$. It is sufficient to prove the following lemma.
\begin{lem}\label{firstlem}
For every $(\vec{a},\vec{b})\in F$ we have $\mathrm{sr}(\vec{a},\vec{b})=-1$, i.e. there exists an autoisometry of $X$ that maps $\vec{a}$ onto $\vec{b}$.
\end{lem}
\noindent\emph{Proof of Lemma \ref{firstlem}.} We split the proof into two parts.\\

\noindent{\bf Step 1} At first, we prove that the set $I=\{(\vec{a},\vec{b})\in F:\mathrm{sr}(\vec{a},\vec{b})=-1\}$ is $\mathfrak{c}$-dense in $F$, i.e. for every basic open set $O$ such that $O\cap F\neq \emptyset$ we have $|O\cap I|=\mathfrak{c}$.

Take an arbitrary basic open $O$ such that $O\cap F\neq \emptyset$. We need to show that $|O\cap I|=\mathfrak{c}$. Let $C\subseteq B$ be the uncountable subset of ordinals such that $X_C=X_B\cap O$. For each $\alpha\in C$, let $\phi_\alpha$ be the winning strategy for Player II in the game $\mathrm{EF}(\vec{a}_\alpha,\vec{b}_\alpha,\alpha)$.

By induction, we shall produce a Cantor scheme $(C_s)_{s\in 2^{<\Nat}}$, where $C_s\subseteq C$ for every $s\in 2^{<\Nat}$, such that for each $s\in 2^{<\Nat}$ we will have $\overline{X_{C_{s0}}}\cap \overline{X_{C_{s1}}}=\emptyset$, and for each $r\in 2^\Nat$ we will have that $\bigcap_n \overline{X}_{C_{r\upharpoonright n}}$ is a singleton. Moreover, we shall produce a function $\psi: 2^{<\Nat}\rightarrow \Nat$, which, similarly as in Proposition \ref{firststep}, will help us define the autoisometries. The argument uses the same ideas as in the proof of Proposition \ref{firststep} thus we will omit some details.

Consider two disjoint uncountable non-stationary subsets $C_0$ and $C_1$ of $C$ such that $\mathrm{diam}(X_{C_i})<1/2$, for $i\in\{0,1\}$, and $\overline{X}_{C_0}\cap \overline{X}_{C_1}=\emptyset$ (where we consider the diameter in some compatible metric, let us say the sum metric on $X^{2n}$, and $\overline{X}_{C_i}$ denotes the closure of $X_{C_i}$). Let $F_i:C_i\rightarrow \omega_1$, for $i\in\{0,1\}$, be some non-stationary function.

Now, for each $\alpha\in C_i$, $i\in\{0,1\}$, consider the response of the strategy $\phi_\alpha$ when Player I plays $(F_i(\alpha),L,t_1)$. Denote such a response by $y_\alpha^i$. Notice that, because $\phi_\alpha$ is a winning strategy for Player II, $d(\vec{a}(m),t_1)=d(\vec{b}(m),y_\alpha^i)$ for any $m<|\vec{a}|$.

Since $F_i[C_i]$ is uncountable and $X^{2n}$ is separable there exists an uncountable subset $D_i\subseteq C_i$ such that $\forall \alpha,\beta\in D_i (d(y_\alpha^i,y_\beta^i)<1/2)$. Let $\psi(i)\in \Nat$ be an arbitrary natural number such that $t_{\psi(i)}$ is within the distance $1/2$ from the set $\{y_\alpha^i:\alpha\in D_i\}$, e.g. $\min\{m\in \Nat: \exists \alpha\in D_i (d(t_m,y_\alpha^i)<1/2)\}$.

Now consider two disjoint uncountable subsets $C_{i0}$ and $C_{i1}$ of $D_i$, again $i\in\{0,1\}$, such that $F_i[C_{ij}]$ is non-stationary, $\mathrm{diam}(X_{C_{ij}})<1/4$, for $j\in\{0,1\}$, and $\overline{X}_{C_{i0}}\cap \overline{X}_{C_{i1}}=\emptyset$. Let $F_{ij}:F_i[C_{ij}]\rightarrow \omega_1$ be some non-stationary function.

For $i,j\in\{0,1\}$ and for each $\alpha\in C_{ij}$ consider the response of the strategy $\phi_\alpha$ when Player I plays $(F_i(\alpha),L,t_1)$ in the first round and then $(F_{ij}\circ F_i(\alpha),R,t_1)$ in the second round. Denote such a response by $x_\alpha^{ij}$. Again notice that, because $\phi_\alpha$, is a winning strategy for Player II, $d(\vec{b}(m),t_1)=d(\vec{a}(m),x_\alpha^{ij})$ for any $m<|\vec{a}|$.

Analogously as above, since $F_{ij}\circ F_i[C_{ij}]$ is uncountable and $X^{2n}$ is separable there exists an uncountable subset $D_{ij}\subseteq C_{ij}$ such that $\forall \alpha,\beta\in D_{ij} (d(x_\alpha^{ij},x_\beta^{ij})<1/4)$. Let $\psi(ij)\in \Nat$ be an arbitrary natural number such that $t_{\psi(ij)}$ is within the distance $1/4$ from the set  $\{x_\alpha^{ij}:\alpha\in D_{ij}\}$, e.g. $\min\{m\in \Nat: \exists \alpha\in D_{ij} (d(t_m,x_\alpha^{ij})<1/4)\}$.

We then again find two disjoint uncountable subsets $C_{ij0}$ and $C_{ij1}$ of $D_{ij}$, $i,j\in\{0,1\}$, such that $F_{ij}\circ F_i[C_{ijk}]$ is a non-stationary subset, $\mathrm{diam}(X_{C_{ijk}})<1/8$, for $k\in\{0,1\}$, and $\overline{X}_{C_{ij0}}\cap \overline{X}_{C_{ij1}}=\emptyset$. Let $F_{ijk}:F_{ij}\circ F_i[C_{ijk}]\rightarrow \omega_1$ be some non-stationary function.\\

Following this scheme and using the same ideas as in the proof of Proposition \ref{firststep} we produce the Cantor scheme $(C_s)_{s\in 2^{<\Nat}}$ and the function $\psi: 2^{<\Nat}\rightarrow \Nat$ such that for each $r\in 2^\Nat$ we have that
\begin{itemize}
\item for every even $n\in \Nat$ and for every $\alpha\in C_{r\upharpoonright n}$
\begin{equation}\label{neweq1}
d(y_\alpha^{r\upharpoonright n},t_{\psi(r\upharpoonright n)})<1/2^{n-1}
\end{equation}
and for every odd $n\in \Nat$ and for every $\alpha\in C_{r\upharpoonright n}$
\begin{equation}\label{neweq2}
d(x_\alpha^{r\upharpoonright n},t_{\psi(r\upharpoonright n)})<1/2^{n-1}
\end{equation}
\item for every $n\in \Nat$
\begin{equation}\label{neweq3}
\mathrm{diam}(X_{C_{r\upharpoonright n}})<1/2^n
\end{equation}
(in the sum metric) 
\item $\bigcap_n \overline{X}_{C_{r\upharpoonright n}}$ is a single pair of tuples $(\vec{a}_r,\vec{b}_r)\in F$

\end{itemize}

The following claim is analogous to Claim \ref{goodpsi}.
\begin{claim}\label{goodpsi2}
Let $r\in 2^\Nat$ and $i,j\in \Nat$. Let $\rho(i)$ be either $\psi(r\upharpoonright 2i)$ or any element from $\psi^{-1}(i)$ provided that $\rho(i)\subseteq r$ and $|\rho(i)|$ is odd. Then we have $$|d(t_i,t_j)-d(t_{\rho(i)},t_{\rho(j)})|<1/2^{2i-1}+1/2^{2j-1}$$ and for every $m<|\vec{a}_r|$ we have $$|d(\vec{a}_r(m),t_j)-d(\vec{b}_r(m),t_{\rho(j)})|<1/2^{2j-3}$$
\end{claim}
\noindent\emph{Proof of the claim.} The proof of the first part, i.e. $$|d(t_i,t_j)-d(t_{\rho(i)},t_{\rho(j)})|<1/2^{2i-1}+1/2^{2j-1}$$ is analogous to the proof of the corresponding part in Claim \ref{goodpsi}.

Let us prove the second part. Suppose that $\rho(j)=\psi(r\upharpoonright 2j)$, the other case is similar. Let $\alpha\in C_{r\upharpoonright 2j}$ be arbitrary. By (\ref{neweq1}) we have $$d(t_{\psi(r\upharpoonright 2j)},y_\alpha^{r\upharpoonright 2j})<1/2^{2j-1}$$ Moreover, by (\ref{neweq3}) we have $$d(\vec{a}_\alpha(m),\vec{a}_r(m))< 1/2^{2j}$$ and $$d(\vec{b}_\alpha(m),\vec{b}_r(m))< 1/2^{2j}$$ Since $\phi_\alpha$ is a winning strategy for Player II, we have $$d(y_\alpha^{r\upharpoonright 2j},\vec{b}_\alpha(m))=d(t_j,\vec{a}_\alpha(m))$$

Putting the (in)equalities above together we get the desired $$|d(\vec{a}_r(m),t_j)-d(\vec{b}_r(m),t_{\psi(r\upharpoonright 2j)})|<1/2^{2j-3}$$

\hfill $\qed$ (of the claim)\\

As in the proof of Proposition \ref{firststep}, for every $r\in 2^\Nat$, we define the autoisometry $\chi_r:X\rightarrow X$ that maps $\vec{a}_r$ onto $\vec{b}_r$ and we are done. For every $x\in X$, we choose arbitrarily some strictly increasing $\iota:\Nat\rightarrow \Nat$ so that the sequence $(t_{\iota(n)})_n$ is a Cauchy sequence converging to $x$. It follows from Claim \ref{goodpsi2} that $(t_{\psi(r\upharpoonright 2\iota(n))})_n$ is Cauchy as well and we set $\chi_r(x)$ as the limit of this sequence. The verification that $\chi_r$ is as desired uses Claim \ref{goodpsi2} in the same way as the proof of Proposition \ref{firststep} uses Claim \ref{goodpsi} that $\chi$ is correctly defined.\\

\noindent{\bf Step 2} We now prove that $I=\{(\vec{a},\vec{b})\in F:\mathrm{sr}(\vec{a},\vec{b})=-1\}$, which was proved to be $\omega_1$-dense in $F$ (even $\mathfrak{c}$-dense), is in fact equal to $F$. The proof is again a variation on the same ideas as in the proofs above so we shall omit some details.\\

Let $R$ be the index set for the set $I=\{(\vec{a},\vec{b})\in F:\mathrm{sr}(\vec{a},\vec{b})=-1\}$, i.e. $I=\{(\vec{a}_\alpha,\vec{b}_\alpha):\alpha\in R\}$. Analogously as in Step 1, for every $P\subseteq R$ by $X_P$ we shall denote the set $\{(\vec{a}_\alpha,\vec{b}_\alpha):\alpha\in P\}$. Moreover, for every $\alpha\in R$ let us denote by $\phi_\alpha$ the winning strategy for Player II in the game $\mathrm{EF}(\vec{a}_\alpha,\vec{b}_\alpha,\infty)$. By induction, we will produce an $\Nat^{<\Nat}$-indexed collection $(R_s)_{s\in \Nat^{<\Nat}}$ such that for every basic open $O$  and $s\in \Nat^{<\Nat}$ we have $X_{R_s}\cap O$ is either empty or uncountable. In addition, we shall again produce a function $\psi:\Nat^{<\Nat}\rightarrow \Nat$ that will help us define the autoisometries.

We describe the general steps of the induction. The $n$-th step of the induction depends on what is $n\; \mathrm{mod}\; 4$. Let us describe the particular cases.\\

Suppose that $n\equiv 1 \;(\mathrm{mod}\; 4)$. Then for every $s\in \Nat^{<\Nat}$ such that $|s|=n-1$ we divide $R_s$ into (not necessarily disjoint) countably many uncountable sets $R_{s1},R_{s2},\ldots$ such that for each $m\in \Nat$ we have $\mathrm{diam}(X_{R_{sm}})<1/2^{|s|/2+1}$ and for every basic open $O$ we have that either $O\cap X_{R_{sm}}=\emptyset$ or $O\cap X_{R_{sm}}$ is uncountable ($X^{2n}$ is second countable). Now for each $m\in \Nat$ and $\alpha\in R_{sm}$, consider the response of the strategy $\phi_\alpha$ when Player I plays successively $(L,t_1),(R,t_1),\ldots,(L,t_{|s|/4+1})$. Denote such a response by $y_\alpha^{sm}$. This finishes the $n$-th step.

In the $(n+1)$-th step, for every $s\in \Nat^{<\Nat}$ such that $|s|=n-2$ and $m\in \Nat$,  since $R_{sm}$ is uncountable and $X^{2n}$ is separable, we can divide $R_{sm}$ into (not necessarily disjoint) countably many uncountable sets $R_{sm1},R_{sm2},\ldots$ such that for each $i\in \Nat$ we have $\mathrm{diam}(\{y_\alpha^{sm}:\alpha\in R_{smi}\})<1/2^{|s|/4+1}$ and for every basic open set $O$ the intersection $O\cap X_{R_{smi}}$ is either empty or uncountable. Let $\psi(smi)\in \Nat$ denote an arbitrary natural number such that $t_{\psi(smi)}$ is within the distance $1/2^{|s|/4+1}$ from the set $\{y_\alpha^{sm}:\alpha\in R_{smi}\}$. This finishes the $(n+1)$-th step.

The $(n+2)$-th, resp. $(n+3)$-th step is similar to the $n$-th, resp. $(n+1)$-th step. The only difference is that Player I plays the right side in his last ($(n/2+2)$-th) round. Briefly, in the $(n+2)$-th step, for every appropriate $s\in \Nat^{<\Nat}$, $m,i\in \Nat$, we divide $R_{smi}$ into (not necessarily disjoint) countably many uncountable sets $R_{smi1},R_{smi2},\ldots$ such that for each $j\in \Nat$ we have $\mathrm{diam}(X_{R_{smij}})<1/2^{|s|/2+2}$ and for every basic open $O$ we have that either $O\cap X_{R_{smij}}=\emptyset$ or $O\cap X_{R_{smij}}$ is uncountable. Then for each $i\in \Nat$ and $\alpha\in R_{smij}$ we consider the response of the strategy $\phi_\alpha$ when Player I plays successively $(L,t_1),(R,t_1),\ldots,(L,t_{|s|/4+1}),(R,t_{|s|/4+1})$. We denote such a response by $x_\alpha^{smij}$. In the next $(n+3)$-th step, we divide each such $R_{smij}$ again into countably many uncountable sets so that for every $k\in \Nat$ we have $\mathrm{diam}(\{x_\alpha^{smij}:\alpha\in R_{smijk}\})<1/2^{|s|/4+2}$. Then $\psi(smijk)\in \Nat$ denote an arbitrary natural number such that $t_{\psi(smijk)}$ is within the distance $1/2^{|s|/4+2}$ from the set $\{x_\alpha^{smij}:\alpha\in R_{smijk}\}$.\\

When the induction is finished, we have produced the $\Nat^{<\Nat}$-indexed collection $(R_s)_{s\in \Nat^{<\Nat}}$ and the partial function $\psi:\Nat^{<\Nat}\rightarrow \Nat$ such that $\psi(s)$, for $s\in \Nat^{<\Nat}$, is defined if and only if $|s|$ is even, such that 
\begin{itemize}
\item for every $s\in \Nat^{<\Nat}$ and basic open $O$ we have that $O\cap X_{R_s}$ is either empty or uncountable
\item for each $v\in \Nat^{\Nat}$ we have that $\bigcap_n \overline{X}_{R_{v\upharpoonright n}}$ is a single pair of tuples $(\vec{a}_v,\vec{b}_v)\in F$
\item for every $s\in \Nat^{<\Nat}$ such that $|s|\equiv 0 \;(\mathrm{mod}\; 4)$ and for every $m,i,j,k\in \Nat$ we have $$\forall \alpha\in R_{smi} (d(t_{\psi(smi)},y_\alpha^{sm})<1/2^{|s|/4+1})$$ and $$\forall \alpha\in R_{smijk} (d(t_{\psi(smijk)},x_\alpha^{smij})<1/2^{|s|/4+2})$$

\end{itemize}

As before, we can then, for every $v\in \Nat^\Nat$, define an autoisometry $\chi_v:X\rightarrow X$ mapping $\vec{a}_v$ onto $\vec{b}_v$ as follows: for any $x\in X$ choose arbitrarily some strictly increasing $\iota:\Nat\rightarrow \Nat$ so that the sequence $(t_{\iota(n)})_n$ is a Cauchy sequence converging to $x$. It then follows, using the same arguments as in Proposition \ref{firststep} and Step 1, that the sequence $(t_{v\upharpoonright 4\iota(n)+3})_n$ is Cauchy as well and we may set $\chi_v(x)$ to be the limit.

It remains to check that every $(\vec{a},\vec{b})\in F$ is of the form $(\vec{a}_v,\vec{b}_v)$ for some $v\in \Nat^{\Nat}$. That follows from the observation that $F=\bigcup_n \overline{X}_{R_n}$ and for every $s\in \Nat^{<\Nat}$ we have $\overline{X}_{R_s}=\bigcup_n \overline{X}_{R_{sn}}$. This finishes the proof of Lemma \ref{firstlem}, which completes the proofs of Proposition \ref{secondstep} and the main theorem.
\end{proof}
\section{Problems}
Let us state few problems here. One of them is to determine the connection between the Scott rank of a countable metric space and its completion. Not surprisingly, the Scott ranks of a countable metric space and its completion may differ.

The proof of the following theorem can be found in \cite{NVT} and is stated as a folklore result there. Recall that a metric space $X$ is ultrahomogeneous if any finite partial isometry between two subspaces of $X$ extends to an isometry of $X$; Friedman et al. in \cite{FFKN} proved that this is equivalent with having the Scott rank $0$.
\begin{thm}\label{NVT_thm}
There is a countable ultrahomogeneous metric space whose completion is not ultrahomogeneous.
\end{thm}

On the other hand, it is a trivial observation that the Scott rank can decrease after the completion. Just consider any non-ultrahomogeneous countable dense subset of an ultrahomogeneous Polish metric space.

However, it is still unclear whether the rank of the completion of a countable metric space can increase arbitrarily after the completion or there is some bound. Note that the latter case would give another proof that the Scott rank of a Polish metric space is countable (assuming the bound is reasonable).

In the questions below, for a countable metric space $X$ we denote by $\overline{X}$ the metric completion.
\begin{question}[Rubin]
Does there exist a function $F:\omega_1\rightarrow \omega_1$ such that for any countable metric space $X$ we have $\mathrm{sr}(\overline{X})\leq F(\mathrm{sr}(X))$?
\end{question}

If the answer were negative, then perhaps the natural way how to show it would be to answer the following question. The positive answer would also generalize Theorem \ref{NVT_thm}.
\begin{question}[Schlicht]
Does there exist, for every countable ordinal $\alpha$, an ultrahomogeneous countable metric space $X_\alpha$ such that $\mathrm{sr}(\overline{X}_\alpha)\geq \alpha$?
\end{question}

When given a concrete Polish metric space the computation of its Scott rank seems to be difficult in general. We address this problem precisely in the next question.
\begin{question}[Zapletal]
Fix a countable ordinal $\alpha$ and a Polish metric space $X$. What is the descriptive set-theoretic complexity of the equivalence relation $\equiv_\alpha$ on $X^m$, for $m\in \Nat$?
\end{question}
Straightforward computation gives that it is at most $\Pi^1_{\alpha+\omega}$ ($\Pi^1_{2n}$ if $\alpha=n\in \omega$). Perhaps, one could hope for better estimates.\\

Since the completness was used essentially in the proof of Theorem \ref{maintheorem}, one can ask whether incomplete spaces can have higher ranks.
\begin{question}[Rubin]
Does there exist a separable metric space $X$ such that $\mathrm{sr}(X)$ is uncountable?
\end{question}

\bigskip

Let us mention the last problem we are interested in here. It is related to the fact mentioned in the introduction that there are Polish metric spaces whose isometry class cannot be described in a Borel way by an $L_{\omega_1 \omega}$ sentence. What are the proper subclasses of the class of all Polish metric spaces in within the isometry classes can be so described? Do these subclasses coincide with the subclasses classifiable by countable structures?

\bigskip

\noindent {\bf Acknowledgment.} The author would like to thank to Matatyahu Rubin and Philipp Schlicht for introducing him into the subject and very fruitful discussions. The author is also grateful to Andr\' e Nies for helpful comments and suggestions.


\begin{thebibliography}{6}
\bibitem{CGK}
J. Clemens, S. Gao, A. Kechris, \emph{Polish metric spaces: their classification and isometry groups}, Bull. Symbolic Logic 7 (2001), no. 3, 361-375
\bibitem{FFKN}
S. Friedman, E. Fokina, M. Koerwien, A. Nies, \emph{Scott analysis of Polish spaces}, preprint
\bibitem{Gao}
S. Gao, \emph{Invariant Descriptive Set Theory}, CRC Press, 2009
\bibitem{Ho}
W. Hodges, \emph{Model theory. Encyclopedia of Mathematics and its Applications, 42}, Cambridge University Press, Cambridge, 1993
\bibitem{Je}
T. Jech, \emph{Set theory. The third millennium edition, revised and expanded}, Springer Monographs in Mathematics. Springer-Verlag, Berlin, 2003
\bibitem{Ke}
H. J. Keisler, \emph{Model Theory for Infinitary Logic: Logic with Countable Conjunctions and Finite Quantifiers}, Studies in Logic and the Foundations of Mathematics, Vol. 62. North-Holland Publishing Co., Amsterdam-London, 1971.
\bibitem{NVT}
L. Nguyen Van Th\' e, \emph{Structural Ramsey theory of metric spaces and topological dynamics of isometry groups}, Mem. Amer. Math. Soc. 206 (2010), no. 968
\bibitem{Ni}
A. Nies, \emph{The complexity of similarity relations for Polish metric spaces}, talk given during the Universality and Homogeneity Trimester, Hausdorff Institute for Mathematics, Bonn (slides available at http://dl.dropbox.com/u/370127/talks/2013/Nies\_HIM\_PolishSpaces.pdf)
\bibitem{Sc}
D. Scott, \emph{Logic with denumerably long formulas and finite strings of quantifiers}, Theory of Models (Proc. 1963 Internat. Sympos. Berkeley), 329--341, North-Holland, Amsterdam, 1965 
\end{thebibliography}
\end{document}